\documentclass[reqno, 12pt]{amsart}
\pdfoutput=1
\usepackage[algoruled,lined,noresetcount,norelsize]{algorithm2e}

\usepackage{graphicx}
\usepackage{amsmath,amssymb,amsthm}
\usepackage{mathrsfs}
\usepackage{mathabx}\changenotsign
\usepackage{dsfont}
\usepackage{xcolor}

\usepackage[T1]{fontenc}
\usepackage{lmodern}
\usepackage[babel]{microtype}
\usepackage[english]{babel}

\usepackage{geometry}
\numberwithin{equation}{section}
\numberwithin{figure}{section}

\usepackage[shortlabels]{enumitem}
\usepackage{caption}
\usepackage{subcaption}

\theoremstyle{plain}
\newtheorem{thm}{Theorem}[section]

\newtheorem{cor}[thm]{Corollary}
\newtheorem{lemma}[thm]{Lemma}
\newtheorem{rem}[thm]{Remark}

\theoremstyle{definition}
\newtheorem{dfn}[thm]{Definition}

\newcommand{\E}{\mathcal{E}}

\usepackage{tikz}
\usetikzlibrary{decorations.pathreplacing}
\usetikzlibrary{decorations.markings}
\usetikzlibrary{bending}

\usepackage[
style=alphabetic,    
isbn=false,                
pagetracker=true, 
maxbibnames=10,            
minbibnames=3,
maxcitenames=3,            
autocite=inline,           
block=space,               
date=short,                
backend=bibtex
]{biblatex}
\renewbibmacro{in:}{}

\begin{document}

\title[S]{Bijective enumeration of rook walks}
\author[Alexander Haupt]{Alexander M. Haupt}
\address{(AH) Technische Universit\"at Hamburg, Institut f\"ur Mathematik, Am Schwarzenberg-Campus 3, 21073 Hamburg, Germany }
\email{alexander.haupt@tuhh.de}

\pagestyle{plain}

\date{\today}

\maketitle

\begin{abstract}
In this paper we answer a question posed by R. Stanley in his collection of Bijection Proof Problems (Problem 240). We present a bijective proof for the enumeration of walks of length $k$ a chess rook can move along on an $m \times n$ board starting and ending on the same square.
\end{abstract}

\section{Introduction}

\begin{dfn}
	Let $S_{m,n,k}$ be the set of walks of length $k$ a chess rook can move along on a rectangular board with width $m$ and height $n$, starting and ending on the bottom left square. 
\end{dfn}

\begin{rem}
	\label{remTorus}
	Note, first of all, that it does not make a difference whether the board is plane or on a torus, by which we mean that opposite edges are identified. This is because the set of allowed moves a rook could take does not change. Secondly we note, that by symmetry of the torus, the starting square does not make a difference to the number of rook walks.
\end{rem}

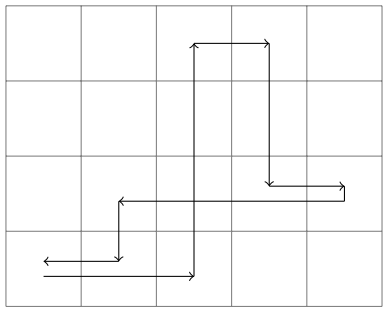
\begin{figure}[h]
	\centering
	\begin{tikzpicture}
	\draw[step=1cm,color=gray] (0,0) grid (5,4);
	\draw[->] (0.5,0.4) -> (2.5,0.4);
	\draw[->] (2.5,0.4) -> (2.5,3.5);
	\draw[->] (2.5,3.5) -> (3.5,3.5);
	\draw[->] (3.5,3.5) -> (3.5,1.6);
	\draw[->] (3.5,1.6) -> (4.5,1.6);
	\draw (4.5,1.6) -> (4.5,1.4);
	\draw[->] (4.5,1.4) -> (1.5,1.4);
	\draw[->] (1.5,1.4) -> (1.5,0.6);
	\draw[->] (1.5,0.6) -> (0.5,0.6);
	\end{tikzpicture}
	\caption{An example with $m=5, n=4, k=8$.}
	\label{fig:example}
\end{figure}

In his collection of Bijection Proof Problems \cite[Problem 240]{Stanley2015BIJECTIVEPP}, Stanley gives the number of rook walks.
\begin{thm}
	\label{theoremStanley}
	We have
	$$|S_{m,n,k}|=\frac{(m+n-2)^k + (n-1)(m-2)^k + (m-1)(n-2)^k + (m-1)(n-1)(-2)^k}{mn}.$$
\end{thm}

Stanley states that the formula is known, but no combinatorial proof of it is. In this paper we provide a bijective proof of Theorem \ref{theoremStanley}, based on the concept of sijections, for which we give a short introduction here. In a recent paper by Fischer and Konvalinka \cite{amsBijection} the concept of signed sets and sijections was introduced. Sijections take the role of bijections for signed sets.

A \emph{signed set} is a finite set $S$ together with a weight function $w:S\rightarrow \{-1,1\}$. Define the weight of the whole set $w(S):=\sum_{s\in S} w(s)$. We write $w_S$ if the signed set is not clear from the context.

A \emph{sijection} $f$ between signed sets $S$ and $T$ is an involution on the set $S \sqcup T$, where $\sqcup$ denotes the disjoint union, such that for~$x \in S \sqcup T$:
$$w(f(x))= \begin{cases}-w(x),&\text{ if }(x \in S\text{ and }f(x)\in S)\text{ or }(x \in T\text{ and }f(x)\in T)\\
w(x),&\text{ otherwise.}\end{cases}$$

\begin{rem}	
	The motivation behind this definition is the following: If we have a sijection between $A$ and $B$ then $w(A)=w(B)$. This is analogous to the fact that if we have a bijection between $A$ and $B$, then $|A|=|B|$. 
	Furthermore, a sijection $A \leftrightarrow B$ is equivalent to a bijection
	$$\{a \in A: w_A(a)=1\} \cup \{b \in B: w_B(b)=-1\} \leftrightarrow \{a \in A: w_A(a)=-1\} \cup \{b \in B: w_B(b)=1\}.$$
\end{rem}

If $A$ is a signed set, we define $-A$ as a copy of $A$, except we have $w_{-A}(a):=-w_{A}(a)$.

If $A$ and $B$ are signed sets, define $A+B$ as the set $A \sqcup B$ together with the weight function $$w_{A+B}(x) := \begin{cases}w_A(x),&\text{if }x\in A\\w_B(x),&\text{if }x\in B.\end{cases}$$

Similarly we define $A-B$ as the signed set $A+(-B)$.

We also define the cross product $A \times B$ of two signed sets $A$ and $B$ to have weight function $w_{A\times B}(x) := w_A(x)\cdot w_B(x)$.

\begin{rem}
	\label{remSubtract}	
	If we have a sijection $A \leftrightarrow B+C$ then we have also a sijection $A - B \leftrightarrow C$. In fact it is the same involution on the set $A \sqcup B \sqcup C$. Also, if we have a sijection $A\leftrightarrow B$ and another sijection $B \leftrightarrow C$, we can infer a sijection $A\leftrightarrow C$, as shown in \cite{amsBijection}.
\end{rem}

We define a simple, but crucial signed set $\alpha_i$ now, which we use it to handle alternating signs. One example is the expression $(-2)^k$ in Theorem \ref{theoremStanley}.
\begin{dfn}
	\label{defAlpha}
	Define $\alpha_i$ to be the signed set containing a single element of weight $(-1)^i$.
\end{dfn}

We start the proof by considering the one-dimensional version of this problem, i.e. rook walks on grids of height $1$. For that we define $S_{m,k} := S_{m,1,k}$. The two-dimensional problem is then related to the one-dimensional problem by the following lemma:
\begin{lemma}
	\label{twoToTwoOne}
	We have a bijection
	$$S_{m,n,k} \longleftrightarrow \sum_{i=0}^k \binom{[k]}{i} \times S_{m,i} \times S_{n,k-i}.$$
\end{lemma}
\begin{proof}
	Let $i$ be the number of horizontal steps of the whole two-dimensional walk. Projecting the two-dimensional walk on the horizontal and vertical axis, we get a pair of walks in $S_{m,i} \times S_{n,k-i}$. The information we lost in this way is exactly which $i$ of the $k$ steps belong to the horizontal walk, which is an element $I \in \binom{[k]}{i}$, an $i$-subset of $[k]$. The result follows.
\end{proof}

We represent elements in $S_{m,k}$ by sequences in $[m-1]^k$ whose total sum is $\equiv 0$ (mod $m$). As noted in Remark \ref{remTorus}, we can imagine the board to wrap around, such that we can interpret each number as the number of squares the rook moves rightwards. We then end on the starting square if and only if the sum of all steps is divisible by $m$.
For example in Figure \ref{fig:example} the walk moves horizontally at positions $\{1,3,5,6,8\}$ and vertical at positions $\{2,4,7\}$, so we have $i=5$. The horizontal subwalk is $(2,1,1,2,4)$ while the vertical subwalk is $(3,2,3)$. Note that these two subwalks sum to $\equiv 0$ (mod $5$) and $\equiv 0$ (mod $4$) respectively.

\begin{lemma}
	\label{OneDim}
	We have a sijection $$[m]\times S_{m,i} \leftrightarrow [m-1]^i + \alpha_i \times [m-1],$$	
	where $S_{m,i}$ is now to be understood as signed set with weight function $w(x) := 1$.
\end{lemma}
\begin{proof}
	
	We consider $i$ even and $i$ odd separately. So we want to show the following bijections
	\begin{enumerate}[(i)]
		\item For all $m$ and even $i$: $[m]\times S_{m,i} \leftrightarrow [m-1]+[m-1]^i$.
		\item For all $m$ and odd $i$: $[m]\times S_{m,i} + [m-1]\leftrightarrow [m-1]^i$.
	\end{enumerate}
	
	We define the set of alternating sequences in $[m-1]^i$ as $$[m-1]^i_{a} := \left \{s \in [m-1]^i : s \text{ is of the form } s=(x,m-x,x,m-x,\ldots) \right\},$$
	and the non-alternating sequences as $$[m-1]^i_{na} := [m-1]^i \setminus [m-1]^i_{a}.$$
	
	Note that we have a simple bijection $[m-1] \leftrightarrow [m-1]^i_a$, as an alternating sequence is determined by its first element.

	Now, for $i=0$, note that there is only one walk of length $0$, so we have
	$$[m]\times S_{m,0} \leftrightarrow [m] \leftrightarrow [m-1] + \{m\} \leftrightarrow [m-1] + [m-1]^0.$$

	Now suppose that $i \ge 1$, take an element in $[m] \times S_{m,i}$ and define a bijection $f$ as follows:
	\begin{enumerate}[(a)]
		\item For $x = (j,(a_1,\ldots,a_i)) \in [m] \times S_{m,i}$ with $j \in [m-1]$, define $f(x):=(j,a_1,\ldots,a_{i-1}) \in [m-1]^i$. This operation maps to exactly those sequences $(b_1,\ldots,b_i) \in [m-1]^i$ with $\sum_{q=2}^i b_q \not \equiv 0$ (mod $m$), because $a_i + \sum_{q=2}^i b_q \equiv \sum_{q=1}^i a_q \equiv 0$ and $a_i \not \equiv 0$ (mod $m$).
		\item For $x = (m, (a_1,\ldots,a_i)) \in [m] \times S_{m,i}$ with $(a_1,\ldots,a_i) \in [m-1]^i_a$, define $f(x):=a_1$. This case is only possible for $i$ even and we map to exactly $[m-1]$.
		\item For $x = (m,(a_1,\ldots,a_i)) \in [m] \times S_{m,i}$, with $(a_1,\ldots,a_i) \in [m-1]^i_{na}$, we do the following: Let $0 \le p < i $ be minimal such that $a_{i-p} + (-1)^p a_1 \not \equiv 0$ (mod $m$), which exists as $(a_1,\ldots,a_i)$ is not alternating. Now define $f(x):=(a_1,\ldots,a_{i-p-1},a_{i-p}+(-1)^p a_1,(-1)^{p-1} a_1,(-1)^{p-2} a_1,\ldots,(-1)^0 a_1)$, where we use modular arithmetic mod $m$. This operation maps to exactly those sequences $(b_1,\ldots,b_i) \in [m-1]^i_{na}$ with $\sum_{q=2}^i b_q \equiv 0$ (mod $m$). 
	\end{enumerate}
	
	For bijection (i), we map to exactly $[m-1]$ from (b) and $[m-1]^i$ from (a) and (c).
	For bijection (ii), we map to exactly $[m-1]^i_{na}$ from (a) and (c). The result follows.
	
\end{proof}

\begin{lemma}
	\label{evenOdd}
	We have a bijection $$[\ell-1]^k + \sum_{i=0:i\text{ odd}}^k \binom{[k]}{i} \times [\ell]^{k-i} \longleftrightarrow \sum_{i=0:i\text{ even}}^k \binom{[k]}{i} \times [\ell]^{k-i}.$$
\end{lemma}
\begin{proof}
	First we define:
	$$[\ell]_i^k := \left \{(S,x) \in \binom{[k]}{i} \times [\ell]^k : \forall j \in S\ x_j = \ell \right \}$$
	for which we have a simple bijection
	$$[\ell]_i^k \longleftrightarrow \binom{[k]}{i} \times [\ell]^{k-i},$$
	that takes an element $(S,x) \in \binom{[k]}{i} \times [\ell]^k$ and removes from $x$ the elements at the positions in $S$. As all of these were copies of $\ell$, we can undo this operation. Now we can restate the lemma: We want to find a bijection
	$$[\ell-1]^k + \sum_{i=0:i\text{ odd}}^k [\ell]_i^k \longleftrightarrow \sum_{i=0:i\text{ even}}^k [\ell]_i^k.$$
	
	Using an idea from Garsia and Milne, as presented in \cite{zeilberger84} by Zeilberger, we define a bijection $f$ as follows:
	\begin{enumerate}[(a)]
		\item For $x \in [\ell-1]^k$, let $f(x)= (\emptyset,x) \in [\ell]_0^k$.
		\item For $(S,x) \in [\ell]_i^k$ with $i$ odd, let $m$ be the smallest index $m$ with $x_m=\ell$, which exists as $i \ge 1$. If $m \in S$ let $f((S,x))= (S\setminus \{m\},x) \in [\ell]_{i-1}^k$. If $m \notin S$ let $f((S,x))= (S\cup \{m\},x) \in [\ell]_{i+1}^k$.
	\end{enumerate}
	To see why this is invertible, we state $f^{-1}$:
	\begin{enumerate}[(a)]
		\item For $(S,x) \in [\ell]_i^k$ with $i$ even and $x$ not containing a copy of $\ell$, we define $f^{-1}((S,x)):=x$.
		\item For $(S,x) \in [\ell]_i^k$ with $i$ even and $x$ containing a copy of $\ell$, let $m$ be the smallest index $m$ with $x_m=\ell$, which exists by assumption. If $m \in S$ let $f^{-1}((S,x))= (S\setminus \{m\},x) \in [\ell]_{i-1}^k$. If $m \notin S$ let $f^{-1}((S,x))= (S\cup \{m\},x) \in [\ell]_{i+1}^k$.
	\end{enumerate}
\end{proof}

\begin{cor}
	\label{mainCor}
	We have a sijection $$\sum_{i=0}^k \binom{[k]}{i} \times \alpha_i \times [\ell]^{k-i} \longleftrightarrow [\ell-1]^k .$$
\end{cor}
\begin{proof}
	This follows from Lemma \ref{evenOdd} directly by Remark \ref{remSubtract} and Definition \ref{defAlpha}.
\end{proof}

\begin{lemma}
	\label{simpleLemma}
	We have a sijection $$\sum_{i=0}^k \binom{[k]}{i} \times \alpha_i \times \alpha_{k-i} \leftrightarrow [2]^k \times \alpha_k$$
\end{lemma}
\begin{proof}
	We have $\alpha_i \times \alpha_{k-i} \leftrightarrow \alpha_k$: Both sides are singleton sets with an element of equal weight whether $k$ and $i$ are odd or even. 
	We also have a bijection
	$$\sum_{i=0}^k \binom{[k]}{i} \leftrightarrow \mathcal{P}([k]) \leftrightarrow [2]^k.$$
	The result follows.
\end{proof}

\begin{thm}
	\label{thmMain}
	We have a bijection
	$$[m]\times [n]\times S_{m,n,k} \leftrightarrow [m+n-2]^k + [n-1]\times [m-2]^k + [m-1]\times [n-2]^k + [m-1]\times [n-1]\times \alpha_k\times [2]^k.$$
\end{thm}
\begin{proof}
	By Lemma \ref{twoToTwoOne}, Lemma \ref{OneDim}, Corollary \ref{mainCor} and Lemma \ref{simpleLemma} we have
	\begin{align*}
	&[m]\times [n]\times S_{m,n,k}\\
	&\leftrightarrow \sum_{i=0}^k \binom{[k]}{i}\times [m] \times S_{m,i} \times [n]\times S_{n,k-i}\\	
	&\leftrightarrow \sum_{i=0}^k \binom{[k]}{i} \times \left( [m-1]^i + \alpha_i \times [m-1]\right)\times \left( [n-1]^{k-i} + \alpha_{k-i} \times [n-1]\right)\\
	&=\sum_{i=0}^k \binom{[k]}{i} \times [m-1]^i \times [n-1]^{k-i}
	+\sum_{i=0}^k \binom{[k]}{i} \times [m-1]^i \times \alpha_{k-i} \times [n-1]\\
	&\quad \quad +\sum_{i=0}^k \binom{[k]}{i} \times \alpha_i \times [m-1] \times [n-1]^{k-i}
	+\sum_{i=0}^k \binom{[k]}{i} \times \alpha_i \times [m-1] \times \alpha_{k-i} \times [n-1]\\
	&\leftrightarrow [m+n-2]^k + [n-1]\times [m-2]^k + [m-1]\times [n-2]^k + [m-1]\times [n-1]\times \alpha_k\times [2]^k.
	\end{align*}
\end{proof}
Theorem \ref{theoremStanley} follows immediately from Theorem \ref{thmMain}, by evaluating the weight of both sides.

\printbibliography

\addcontentsline{toc}{section}{Bibliography}

\end{document}